\documentclass[11pt,a4paper]{article}
\usepackage{amsmath,amssymb,amsthm,amsfonts,latexsym,graphicx,subfigure}

\usepackage[all,import]{xy}
\usepackage[numbers,sort&compress]{natbib}
\usepackage{indentfirst}
\usepackage{fancyhdr,enumerate}
\usepackage{bbm,color}

\usepackage{accents,cases}
\usepackage{multirow}
\usepackage[lined, ruled, linesnumbered, longend]{algorithm2e}

\usepackage{array,float}
\newcommand{\PreserveBackslash}[1]{\let\temp=\\#1\let\\=\temp}
\newcolumntype{C}[1]{>{\PreserveBackslash\centering}p{#1}}
\newcolumntype{R}[1]{>{\PreserveBackslash\raggedleft}p{#1}}
\newcolumntype{L}[1]{>{\PreserveBackslash\raggedright}p{#1}}

\DeclareMathOperator*{\vol}{\ensuremath{Vol}}
\DeclareMathOperator*{\RQ}{\ensuremath{RQ}}

\newcommand{\id}{\mathrm{Id}}
\newcommand{\diag}{\mathrm{diag}}
\newcommand{\R}{\ensuremath{\mathbb{R}}}

\theoremstyle{plain}
\newtheorem{theorem}{Theorem}[section]
\newtheorem{lemma}[theorem]{Lemma}
\newtheorem{proposition}[theorem]{Proposition}
\newtheorem{corollary}[theorem]{Corollary}
\theoremstyle{definition}

\newtheorem{ex}[theorem]{Example}
\theoremstyle{remark}

\usepackage[affil-it]{authblk}

\begin{document}
	\bibliographystyle{unsrt} 
	\title{Coloring the normalized Laplacian for oriented hypergraphs}
	\author[1,2,3]{Aida Abiad}
	\author[4]{Raffaella Mulas}
	\author[4]{Dong Zhang}
	\affil[1]{Eindhoven University of Technology, Eindhoven, The Netherlands}
	\affil[2]{Ghent University, Ghent, Belgium}
	\affil[3]{Vrije Universiteit Brussels, Brussels, Belgium}
	\affil[4]{Max Planck Institute for Mathematics in the Sciences, Leipzig, Germany}
	\date{}

	\maketitle
	
	\begin{abstract}The independence number, coloring number and related parameters are investigated in the setting of oriented hypergraphs using the spectrum of the normalized Laplace operator. For the independence number, both an inertia--like bound and a ratio--like bound are shown. A Sandwich Theorem involving the clique number, the vector chromatic number and the coloring number is proved, as well as a lower bound for the vector chromatic number in terms of the smallest and the largest eigenvalue of the normalized Laplacian. In addition, spectral partition numbers are  studied in relation to the coloring number.
		\vspace{0.2cm}
		
	\end{abstract}
	
	\allowdisplaybreaks[4]
	\section{Introduction}
	Oriented hypergraphs were introduced by Shi in \cite{Shi92} as a generalization of classical hypergraphs in which a plus or minus sign is assigned to each vertex--hyperedge incidence. Since their introduction, such hypergraphs have received a lot of attention. The adjacency and Kirchhoff Laplacian matrices of oriented hypergraphs were introduced by Reff and Rusnak \cite{ReffRusnak}, while the normalized Laplacian was introduced by Jost together and the second author of this paper \cite{Hypergraphs}. It is known that the spectra of these matrices encode many qualitative properties of the associated oriented hypergraph, and several problems in spectral hypergraph theory arise when trying to generalize the classical spectral results that are known for graphs. We refer the reader to \cite{CRRY2015,DR2019,GRR2019,GR2020,KR2019,Reff2014,Reff2016,ReffRusnak,RRSS2017,Rusnak2013,CLRRW2018} for a vast\,---\,but not complete\,---\,literature on the adjacency and Kirchhoff Laplacian matrices for oriented hypergraphs. We refer to \cite{AndreottiMulas,Hypergraphs,pLaplacians1,MulasZhang,Sharp,Classes,MKJ} for some literature on the hypergraph normalized Laplacian.

	The overall aim of this paper is to bring forward the study of the spectrum of the normalized Laplacian of oriented hypergraphs. In particular, this paper investigates its relation with parameters that depend on the structural properties of the hypergraphs, such as the coloring number and the independence number. While these relations have already been partly investigated in \cite{MulasZhang,pLaplacians1} for the coloring number, to the best of our knowledge they have not been yet studied for the independence number. The \emph{independence number of a hypergraph} \cite{hyp-ind10} is the maximum size of a set of vertices such that, for each pair of vertices in this set, there is no hyperedge containing both of them. The \emph{coloring number of a hypergraph} was defined by Erdős and Hajnal in 1966 \cite{Erdos} as the minimal number of colors needed for coloring the vertices so that, if two vertices are contained in a common hyperedge, they receive different colors.
	
	The following two bounds are well-known for the order of an independent set in a graph. Let $G$ be a graph with $n$ vertices and adjacency matrix eigenvalues $\theta_1\leq \ldots \leq \theta_n$. The first well-known spectral bound ({\em `inertia bound'\/}) for the independence number $\alpha$ of $G$ is due to Cvetkovi\'c \cite{C71}:
	
	\begin{equation*}
	\alpha(G)\le \min  \{\#\{i : \theta_i\le 0\},\#\{i : \theta_i\ge 0\} \}.
	\end{equation*}
	
	When $G$ is regular, another well-known bound ({\em `ratio bound'\/}) is due to Hoffman (unpublished, see for instance \cite{BH2012}):
	\begin{equation*}
	\label{bound:hoffman}
	\alpha(G) \leq n\frac{-\theta_1}{\theta_n-\theta_1}.
	\end{equation*}

	In this paper we present the first inertia--like  bound and  ratio--like bound for the independence number of an oriented hypergraph. Both upper bounds involve the eigenvalues of the normalized Laplacian and are the analogous of the aforementioned celebrated bounds for the adjacency spectrum of a graph. In addition, we study the coloring number of a hypergraph in relation to the clique number, to the vector chromatic number and to the spectral partition numbers of a hypergraph.
	
	This paper is structured as follows. In Section \ref{section:prel} we offer an overview of the basic definitions and notations that will be needed throughout this paper. In Section \ref{section:inertia} we prove a hypergraph--Laplacian version of the inertia bound. Similarly, in Section \ref{section:ratio}, we show a hypergraph--Laplacian version of the ratio bound. In Section \ref{section:vector} we investigate the coloring number and the vector chromatic number of oriented hypergraphs and finally, in Section \ref{section:spectralpart}, we show that the introduced spectral partition numbers are closely related to the coloring number.
	
	\section{Preliminary definitions and notations}\label{section:prel}
	An \emph{oriented hypergraph} \cite{Shi92} is a triple $\Gamma=(V,H,\psi_\Gamma)$ such that $V$ is a finite set of vertices, $H$ is a finite multiset of elements $h\in \mathcal{P}(V)\setminus\{\emptyset\}$ called \emph{hyperedges}, while $\psi_\Gamma:(V,H)\rightarrow \{-1,0,+1\}$ is the \emph{incidence function} and it is such that 
	\begin{equation*}
	\psi_\Gamma(i,h)\neq 0 \iff i\in h.
	\end{equation*}A vertex $i$ is an \emph{input} for a hyperedge $h$ if $\psi_\Gamma(i,h)=1$, and an \emph{output} if $\psi_\Gamma(i,h)=-1$. Two vertices $i\neq j$ are \emph{co-oriented in a hyperedge $h$} if $\psi_\Gamma(i,h)=\psi_\Gamma(j,h)\neq 0$, and they are \emph{anti-oriented in $h$} if $\psi_\Gamma(i,h)=-\psi_\Gamma(j,h)\neq 0$. Given a hyperedge $h$, we denote by $h_{\textrm{in}}$ the set of its inputs and by $h_{\textrm{out}}$ the set of its outputs. Clearly, $h=h_{\textrm{in}}\cup h_{\textrm{out}}$.
	
	The \emph{degree} of a vertex $i$, denoted $\deg(i)$, is the number of hyperedges containing $i$. The \emph{size} of a hyperedge $h$, denoted $\# h$, is the number of vertices that are contained in $h$. We say that a hypergraph is \emph{$d$--regular} if all vertices have degree $d$. 
	
	Observe that simple graphs can be seen as oriented hypergraphs such that $H$ is a set and, for each $h\in H$, there exists a unique $i\in V$ that is an input for $h$ and there exists a unique $j\in V$ that is an output for $h$. More generally, signed graphs can be seen as oriented hypergraphs such that $H$ is a set and each hyperedge has size $2$ \cite{ReffRusnak}.\newline
	
	Throughout the paper $\Gamma=(V,H,\psi_\Gamma)$ is an oriented hypergraph on $n$ vertices $\{1,\ldots,n\}$ and $m$ hyperedges $\{h_1,\ldots, h_m\}$. For simplicity, we assume that there are no vertices of degree zero.\newline
	
	Given a subset $S\subset V$, we define the \emph{sub-hypergraph} $\Gamma|_S$ as the triple $(S,H|_S,\psi_{\Gamma}|_{S})$, where
	\begin{equation*}
	H|_S:=\{h\cap S:h\in H\}.
	\end{equation*}
	
	The \emph{degree matrix} of $\Gamma$ \cite{ReffRusnak} is the $n\times n$ diagonal matrix
	\begin{equation*}
	D=D(\Gamma):=\diag(\deg(1),\ldots,\deg(n)).
	\end{equation*}
	The \emph{adjacency matrix} of $\Gamma$ \cite{ReffRusnak} is an $n\times n$ matrix $A=A(\Gamma):=(A_{ij})$, with entries $A_{ii}:=0$ for each $i\in V$ and, for $i\neq j$,
	\begin{align*}
	A_{ij}:=& \# \{\text{hyperedges in which }i \text{ and }j\text{ are anti-oriented}\}\\
	&-\# \{\text{hyperedges in which }i \text{ and }j\text{ are co-oriented}\}.
	\end{align*} 
	
	The \emph{normalized Laplacian matrix} of $\Gamma$ \cite{Hypergraphs} is the $n\times n$ matrix 
	\begin{equation*}
	L=L(\Gamma):=\id-D^{-1}A,
	\end{equation*}
	where $\id$ is the $n\times n$ identity matrix. It is known that $L$ has $n$ real, nonnegative eigenvalues, counted with multiplicity \cite{Hypergraphs}. We denote them by $\lambda_1\leq\ldots\leq \lambda_n.$ Moreover, by the Courant--Fischer--Weyl min--max principle \cite[Theorem 36]{Hypergraphs} such eigenvalues are the min--max of the \emph{Rayleigh quotient}, defined for a nonzero function $f:V\rightarrow \mathbb{R}$ by
	\begin{equation*}
	\RQ(f):=\frac{\sum_{h\in H}\left(\sum_{i\in h_{\textrm{in}}}f(i)-\sum_{j\in h_{\textrm{out}}}f(j)\right)^2}{\sum_{i\in V}\deg(i)f(i)^2}.
	\end{equation*}In particular,
	\begin{equation*}
	\lambda_1=\min_{f}\RQ(f)\quad\text{and}\quad \lambda_n=\max_{f}\RQ(f).
	\end{equation*} To prove our theorems, we will use the above characterization of $\lambda_1$ and $\lambda_n$.
	
	\section{Inertia--like bound for the independence number}\label{section:inertia}
	The \emph{independence number} $\alpha$ of a graph $G$ is the maximum size of a set of vertices that are not pairwise adjacent. The following upper bound, due to Cvetkovi\'c \cite{C71}, is a classic result in spectral graph theory. Let $G$ be a graph with adjacency spectrum $\theta_1\leq \ldots \leq \theta_n$. Then
	\begin{equation}
	\label{bound:cvetkovic}
	\alpha(G)\le \min  \{\#\{i : \theta_i\le 0\},\#\{i : \theta_i\ge 0\} \}.
	\end{equation}
	The inequality \eqref{bound:cvetkovic} is often referred to as the \emph{Cvetković bound} or \emph{inertia bound}. This bound was initially shown for the adjacency matrix, but it is easy to see that it holds for any  weighted adjacency matrix. A variant of the inertia bound has been recently used by Huang \cite{H2020} to prove a long standing conjecture in computer science.
	
	In this section we propose two generalizations of the independence number for oriented hypergraphs and we prove a generalized inertia bound in terms of the eigenvalues of the normalized Laplacian. We refer to \cite{hyp-ind1,hyp-ind2,hyp-ind3,hyp-ind4,hyp-ind5,hyp-ind6,hyp-ind7,hyp-ind8,hyp-ind9,hyp-ind10} for references on the independence number for classical hypergraphs, with no relation to the spectral properties. The list is by no mean complete, but it gives a good overview of the work done in this direction. Moreover, we refer to \cite{bdz2005,SunDas2020} for a selection of references on the relation between the independence number of a graph and the spectra of its associated operators.
	
	We say that a subset $U\subseteq V$ is \emph{independent} if $\#(U\cap h)\le 1$ for all $h\in H$.  An independent set $U$ is a \emph{maximal independent set} if, for all $i\in V\setminus U$, $U\cup\{i\}$ is not independent. Following \cite{hyp-ind10}, we define the \emph{independence number} of $\Gamma$ as 
	\begin{equation*}
	\alpha(\Gamma):=\max\{\#U:U\subseteq V \text{ independent}\}.
	\end{equation*}
	Similarly, we say that a subset $U\subseteq V$ is \emph{weakly independent} if $A_{ij}=0$ for all $i,j\in U$. We define the \emph{weak independence number} of $\Gamma$ as
	\begin{equation*}
	\alpha_w(\Gamma):=\max\{\#U:U\subseteq V \text{ weakly independent}\}.
	\end{equation*}
	Note that, while the independence number $\alpha$ of $\Gamma$ depends on $V$ and $H$ but not on the incidence function $\psi_\Gamma$, the weak independence number $\alpha_w$ of $\Gamma$ also depends on $\psi_\Gamma$.

	Next we state our first main result, an inertia--like bound for the hypergraph normalized Laplacian which upper bounds both $\alpha$ and $\alpha_w$.
	
	\begin{theorem}[Inertia--like bound for the normalized Laplacian of an oriented hypergraph]\label{thm:Inertia-Laplacian}
		Let $\Gamma$ be an oriented hypergraph with eigenvalues of the normalized Laplacian matrix $\lambda_1\leq\ldots\leq \lambda_n$. Then
		
		\begin{equation}\label{eq:inertia}
		\alpha(\Gamma) \le \alpha_w(\Gamma)\leq \min\{\#\{i:\lambda_i\leq 1\},\#\{i:\lambda_i\geq 1\}\}.
		\end{equation}
	\end{theorem}
	\begin{proof}The first inequality in \eqref{eq:inertia} is clear since, if $U\subseteq V$ is independent, then $A_{ij}=0$ for all $i,j\in U$, therefore $U$ is also weakly independent.
		
		In order to prove the second inequality, let $U\subseteq V$ be a weakly independent set such that $\# U=\alpha_w$. Then, the matrix obtained from $L$ by deleting the rows and columns corresponding to the vertices in $V\setminus U$ is the identity matrix of size $n-\alpha_w$ and it has eigenvalue $1$ with multiplicity $n-\alpha_w$. By  Cauchy Interlacing Theorem (see for instance [\cite{matrixanalysis}, Theorem 4.3.17]), this implies that 
		\begin{equation*}
		\#\{i:\lambda_i> 1\}\le n-\alpha_w\quad \text{and}\quad \#\{i:\lambda_i< 1\}\le n-\alpha_w.
		\end{equation*}Therefore,
		\begin{equation*}
		\alpha_w\leq \min\{n-\#\{i:\lambda_i> 1\},n-\#\{i:\lambda_i< 1\}\}=\min\{\#\{i:\lambda_i\leq 1\},\#\{i:\lambda_i\geq 1\}\}.
		\end{equation*}
	\end{proof}
	The next example illustrates the sharpness of the upper bound in Theorem \ref{thm:Inertia-Laplacian}.
	
	\begin{ex}\label{example:independent}
		If $\Gamma=(V,H,\psi_\Gamma)$ is such that $V=\{1,\ldots,n\}$ and $H=\{\{1\},\ldots,\{n\}\}$ (Figure \ref{fig:loops}), then clearly $\lambda_1=\ldots=\lambda_n=1$ and $\alpha=\alpha_w=n$, independently of $\psi_\Gamma$. Thus, the bound \eqref{eq:inertia} is sharp.
		\begin{figure}[h]
			\centering
			\includegraphics[width=10cm]{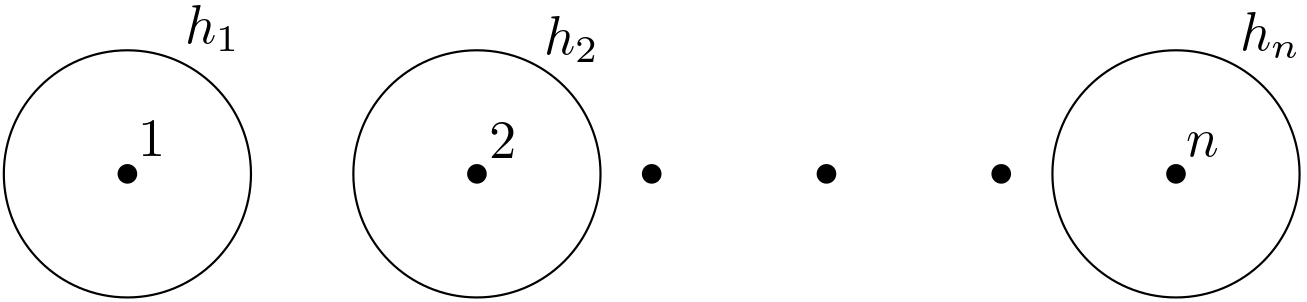}
			\caption{Hypergraph which holds the bound \eqref{eq:inertia} with equality.}
			\label{fig:loops}
		\end{figure}
	\end{ex}

	\section{Ratio--like bound for the independence number}\label{section:ratio}
	In this section we prove a ratio--like bound for the independence number of an oriented hypergraph. For a regular graph $G$ with degree $k$ and adjacency spectrum $\theta_1\leq \ldots \leq \theta_n=k$, Hoffman \cite[p.39]{BH2012} proved the following well-known bound (\emph{Hoffman's ratio bound}): 
	\begin{equation}\label{ratiobound}
	\alpha (G)\leq n\frac{-\theta_1}{\theta_n-\theta_1}
	\end{equation}
	and if an independent set $U$ meets \eqref{ratiobound} then every vertex not in $U$ is adjacent to precisely $-\theta_1$ vertices of $U$.
	
	Our next theorem shows a ratio--like bound for the general case of regular oriented hypergraphs, and for the eigenvalues of $L$ instead of the eigenvalues of $A$.
	\begin{theorem}[Ratio--like bound for the normalized Laplacian of an oriented hypergraph]\label{thm:ratio} Let $\Gamma$ be a $d$--regular oriented hypergraph such that $\#h_{in}=\#h_{out}$ for all $h\in H$. Let $\lambda_1\leq\ldots\leq \lambda_n$ be the eigenvalues of the normalized Laplacian matrix of $\Gamma$. Then,
		\begin{equation}\label{eq:ratio}
		\alpha(\Gamma)\le n\Bigl(1-\frac{1}{\lambda_n}\Bigr).
		\end{equation}
		If equality holds, then $\alpha(\Gamma) \le n/2$ and in particular, if $\alpha(\Gamma) =n/2$, it implies that $\Gamma$ is a bipartite graph with $\alpha$ vertices on each side of the bipartition. Moreover, if \eqref{eq:ratio} holds with equality and $\Gamma$ is a graph, then $\Gamma|_{V\setminus U}$ is a $d(n-2\alpha)/(n-\alpha)$--regular graph.
	\end{theorem}
	\begin{proof}
		Let $U$ be a maximal independent set, so that $\# U=\alpha$, and let $f:V\to \R$ be such that
		\begin{equation*}
		f(i):=\begin{cases}
		t&\text{ if } i\in U,\\
		1&\text{ if } i\in V\setminus U.
		\end{cases}
		\end{equation*}
		Let $\vol(U):=\sum_{i\in U}\deg(i)$. Then
		\begin{equation*}
		\RQ(f)=\frac{\sum_{h\in H,h\cap U\ne\emptyset}(t-1)^2}{t^2\sum_{i\in U}\deg(i)+\sum_{i\in V\setminus U}\deg(i)}=\frac{\vol(U)(t-1)^2}{
			\vol(U)t^2+\vol(V)-\vol(U)}=\frac{\alpha(t-1)^2}{
			\alpha t^2+n-\alpha}=:\phi(t),
		\end{equation*}
		which attains its maximum at $t=1-\frac{n}{\alpha}$. Consequently, the largest eigenvalue is such that
		\begin{equation*}
		\lambda_n\ge \phi\Bigl(1-\frac{n}{\alpha}\Bigr)= \frac{n}{n-\alpha}.
		\end{equation*}Hence,
		\begin{equation*}
		\alpha\le n\Bigl(1-\frac{1}{\lambda_n}\Bigr).
		\end{equation*}This proves \eqref{eq:ratio}. 
		
		If equality holds, then 
		\begin{equation*}
		f(i):=\begin{cases}
		1-\frac{n}{\alpha}&\text{ if } i\in U,\\
		1&\text{ if } i\in V\setminus U.
		\end{cases}
		\end{equation*}is an eigenfunction with eigenvalue $\lambda_n=\frac{n}{n-\alpha}$. Since $\Gamma$ is $d$--regular, this implies that $f$ is an eigenfunction for $A$ with eigenvalue $d(1-n/(n-\alpha))=d\alpha/(\alpha-n)$. Therefore, for all $i\in V\setminus U$,
		\begin{equation*}
		\sum_{j\in U}A_{ij}\Bigl(1-\frac{n}{\alpha}\Bigr)+ \sum_{j'\in V\setminus U}A_{ij'}=\frac{d\alpha}{\alpha-n}.
		\end{equation*}Also, since $\Gamma$ is $d$--regular and $\#h_{in}=\#h_{out}$ for all $h$, we have that, for each $i\in V\setminus U$, $\sum_{j\in V}A_{ij}=d$. Hence,
		\begin{equation*}
		\sum_{j'\in V\setminus U}A_{ij'}=d-\sum_{j\in U}A_{ij}
		\end{equation*}and therefore
		\begin{equation*}
		-\frac{n}{\alpha}\Biggl(\sum_{j\in U}A_{ij}\Biggr)=\frac{d\alpha}{\alpha-n}-d=\frac{nd}{\alpha-n}.\end{equation*}Now, since $U$ is an independent set,
		\begin{equation*}
		\sum_{j\in U}A_{ij}=\#\{h\ni i:h\cap U\text{ and }i\text{ anti-oriented in }h\}-\#\{h\ni i:h\cap U\text{ and }i\text{ co-oriented in }h\}.
		\end{equation*}Therefore
		\begin{align*}
		d&\geq \#\{h\ni i:h\cap U\text{ and }i\text{ anti-oriented in }h\}+\#\{h\ni i:h\cap U\text{ and }i\text{ co-oriented in }h\}\\
		&\geq 
		\#\{h\ni i:h\cap U\text{ and }i\text{ anti-oriented in }h\}-\#\{h\ni i:h\cap U\text{ and }i\text{ co-oriented in }h\}\\
		& = \frac{\alpha d}{n-\alpha},
		\end{align*}which implies $\alpha\leq n/2$. In particular, if $\alpha= n/2$, then $i$ and $h\cap U$ are anti-oriented, for all $h\in H$ and for all $i\in V\setminus U$. Clearly, this implies that $\#h_{in}=\#h_{out}=1$ for all $h$, hence $\Gamma$ is a graph. Moreover, by construction $\Gamma$ must be a bipartite graph with $\alpha$ vertices on each side of the bipartition.
		
		Finally, if we assume that \eqref{eq:ratio} is attained with equality and $\Gamma$ is a graph, by the above calculations we must have that, for all $i\in V\setminus U$, it holds that
		\begin{equation*}
		\#\{h\ni i: h\cap U\neq \emptyset\}=\alpha  d/(n-\alpha)\quad\text{and}\quad \#\{h\ni i: h\cap U= \emptyset\}=d-\alpha  d/(n-\alpha)=\frac{d(n-2\alpha)}{n-\alpha}.
		\end{equation*}Hence, the subgraph $\Gamma|_{V\setminus U}$ is a $d(n-2\alpha)/(n-\alpha)$--regular graph.
	\end{proof}
	
	The following examples illustrate that the bound \eqref{eq:ratio} from Theorem \ref{thm:ratio} is best possible.
	
	\begin{ex}
		Let $\Gamma=(V,H,\psi_\Gamma)$ be such that:
		\begin{itemize}
			\item $V=\{1,2,3,4\}$;
			\item $H=\{h_1,h_2,h_3\}$;
			\item $h_1$ has $1$ and $2$ as inputs, $3$ and $4$ as outputs;
			\item $h_2$ has $1$ and $3$ as inputs, $2$ and $4$ as outputs;
			\item $h_3$ has $1$ and $4$ as inputs, $2$ and $3$ as outputs.
		\end{itemize}Then, $\Gamma$ is  $3$--regular and it satisfies the condition of Theorem \ref{thm:ratio}. Also, $\alpha(\Gamma)=1$ and $\lambda_4=4/3$, implying that
		\begin{equation*}
		\alpha(\Gamma)=1=n\left(1-\frac{1}{\lambda_n}\right).
		\end{equation*}Therefore, this $\Gamma$ provides an example of when the bound \eqref{eq:ratio} is sharp.
	\end{ex}
	\begin{ex}\label{ex:4}
		Let $\Gamma=(V,H,\psi_\Gamma)$ be such that (Figure \ref{fig:4}):
		\begin{itemize}
			\item $V=\{1,2,3,4\}$;
			\item $H=\{h_1,h_2,h_3\}$;
			\item $h_1$ has $1$ as input and $2$ as output;
			\item $h_2$ has $3$ as input and $4$ as output;
			\item $h_3$ has $1$ and $2$ as inputs, $3$ and $4$ as outputs.
		\end{itemize}Then, $\Gamma$ is $2$--regular and it satisfies the condition of Theorem \ref{thm:ratio}. Also, $\alpha(\Gamma)=1$ and $\lambda_n=2$, therefore,
		\begin{equation*}
		\alpha (\Gamma)=1<n\left(1-\frac{1}{\lambda_n}\right)=2.
		\end{equation*}
		This example shows that the bound \eqref{eq:ratio} cannot be improved by using equality.
		\begin{figure}[h]
			\centering
			\includegraphics[width=9cm]{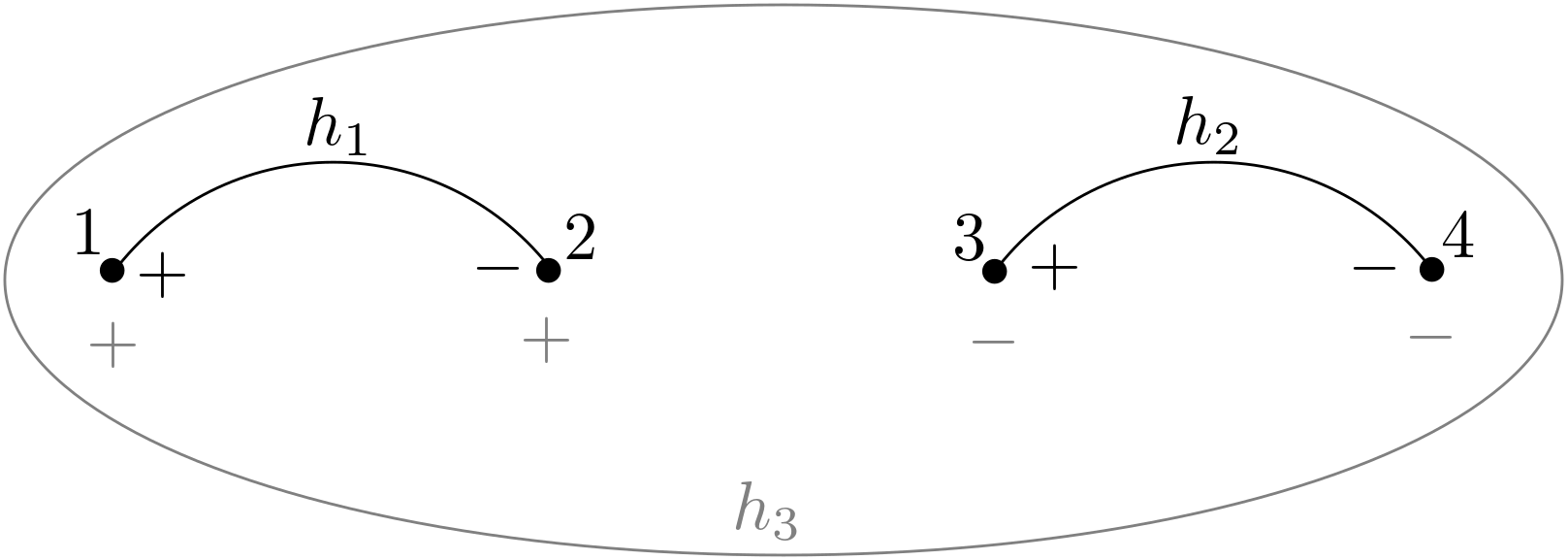}
			\caption{The oriented hypergraph from Example \ref{ex:4}.}
			\label{fig:4}
		\end{figure}
	\end{ex}

	\section{Coloring number and Sandwich Theorem}\label{section:vector}

	In this section we introduce the concepts of coloring number, clique number and vector chromatic number of an oriented hypergraph with the aim to prove a Sandwich Theorem type of result involving these parameters. We refer the reader to \cite{GR2001} for the basic background on these parameters for the graph case, and to \cite{Knuth} for the Sandwich Theorem for the graph case.
	
	A \emph{proper $k$--coloring of the vertices} of an oriented hypergraph $\Gamma$ is a function $f:V\to \{1,\ldots,k\}$ such that, for each hyperedge $h$, if $i\neq j$ are both in $h$, then $f(i)\ne f(j)$. The \emph{coloring number (or chromatic number)} $\chi=\chi(\Gamma)$ is the minimal $k$ such that there exists a proper $k$--coloring of the vertices of $\Gamma$.
	
	Similarly, we define the \emph{clique number} $\omega=\omega(\Gamma)$ as the size of the largest $U\subset V$ such that for any $i\ne j\in U$ it holds that $\{i,j\}\subset h$ for some $h\in H$.  Note that if $\Gamma$ is a graph, then the clique number is just the size of the largest clique of $\Gamma$.
	
	We adapt the concept of vector chromatic number of a graph  \cite{KMS98} to the setting of oriented hypergraphs as follows. We define the \emph{vector chromatic number} of an oriented hypergraph $\Gamma$  $\chi_{\textrm{v}}=\chi_{\textrm{v}}(\Gamma)$ as the minimal $k$ for which there exists an assignment of unit vectors $\mathbf{u_i}\in \mathbb{R}^n$ to each vertex $i\in V$, such that
	\begin{equation}\label{eq:vector}
	\langle \mathbf{u_i},\mathbf{u_j}\rangle= -\frac{1}{k-1}
	\end{equation}whenever $i\ne j\in h$ for some $h\in H$.
	
	Now we are ready to prove a \emph{Sandwich Theorem} for the clique number, vector chromatic number and coloring number of an oriented hypergraph.
	
	\begin{theorem}\label{thm:sandwich}[Sandwich Theorem]
		For any oriented hypergraph $\Gamma$, $$\omega (\Gamma) \leq \chi_{\mathrm{v}}(\Gamma)\leq \chi (\Gamma).$$
	\end{theorem}
	\begin{proof}In order to prove the inequality on the left hand side, assume by contradiction that $\chi_{\textrm{v}}\le  \omega-1$ and let $U\subset V$ be a clique of size $\omega$. By definition of vector chromatic number, for any pair of vertices $i\ne j$ in $U$, $\langle \mathbf{u_i},\mathbf{u_j}\rangle=-1/(\chi_{\textrm{v}}-1)$. Therefore, from the assumption that $\chi_{\textrm{v}}\le  \omega-1$, it follows that
		$$\|\sum_{i\in U}\mathbf{u_i}\|^2=\sum_{i\in U} \|\mathbf{u_i}\|^2+\sum_{i\ne j\text{ in }U}\langle \mathbf{u_i},\mathbf{u_j}\rangle=\omega+2{\omega\choose 2}\Biggl(-\frac{1}{\chi_{\textrm{v}}-1}\Biggr)\le-\frac{-\omega}{\omega-2}<0,$$
		which is a contradiction.
		
		In order to prove the inequality on the right hand side observe that, since $\chi\le n$, by Lemma 4.1 in \cite{KMS98} there exist $\chi$ unit vectors $\mathbf{\hat{u}_1},\ldots,\mathbf{\hat{u}_\chi}\in\R^n$ satisfying 
		\begin{equation*}
		\langle \mathbf{\hat{u}_i},\mathbf{\hat{u}_j}\rangle= -\frac{1}{\chi-1} \quad \text{whenever } 1\leq i\neq j\leq \chi.
		\end{equation*}Now, for each coloring class $C_l$ of $\Gamma$, let $\mathbf{u_i}:=\mathbf{\hat{u}_l}$ if $i\in C_l$. Then, $\mathbf{u_1},\ldots,\mathbf{u_n}$ are unit vectors that satisfy \eqref{eq:vector}. Therefore $\chi_{\textrm{v}}\leq \chi$.
	\end{proof}

	Our next main result (Theorem \ref{thm:vector-chromatic}) presents a lower bound for the vector chromatic number of an oriegnted hypergraph in terms of the smallest and the largest eigenvalue of the normalized Laplacian matrix. First we prove a preliminary result that we will need.
	\begin{lemma}\label{lemma:Chung}
		If $\mathbf{v_1},\ldots,\mathbf{v_n}\in\mathbb{R}^n$, then
		$$\lambda_1\leq \frac{\sum_{i,j}\sqrt{\frac{\deg(i)}{\deg(j)}}L_{ij}\langle \mathbf{v_i},\mathbf{v_j}\rangle}{\sum_i \|\mathbf{v_i}\|^2} \leq \lambda_n.$$
	\end{lemma}
	
	\begin{proof}Note that the \emph{Chung Laplacian} $\mathcal{L}:=D^{\frac12}LD^{-\frac12}$ is a symmetric matrix that has the same spectrum as $L$. Since $\mathcal{L}$ is symmetric, as a consequence of the Min--Max Principle we have that
		\begin{equation*}
		\lambda_1\le \frac{\sum_{i,j}\mathcal{L}_{ij}\langle \mathbf{v_i},\mathbf{v_j}\rangle }{\sum_i \|\mathbf{v_i}\|^2}\le \lambda_n
		\end{equation*}for each $\mathbf{v_1},\ldots,\mathbf{v_n}\in\mathbb{R}^n$. Hence,
		\begin{equation*}
		\lambda_1\le \frac{\sum_{i,j}\sqrt{\frac{\deg(i)}{\deg(j)}}L_{ij}\langle \mathbf{v_i},\mathbf{v_j}\rangle }{\sum_i \|\mathbf{v_i}\|^2}\le \lambda_n.
		\end{equation*}
	\end{proof}
	
	Now we are ready to prove the lower bound for the vector chromatic number.
	\begin{theorem}
		\label{thm:vector-chromatic}Let $\Gamma$ be an oriented hypergraph with eigenvalues of the normalized Laplacian matrix $\lambda_1\leq\ldots\leq \lambda_n$. Then,
		\begin{equation*} 
		\chi_{\mathrm{v}}(\Gamma)\ge\frac{\lambda_n-\lambda_1}{\min\{\lambda_n-1,1-\lambda_1\}}.
		\end{equation*}
	\end{theorem}
	\begin{proof}
		Let $\mathbf{u_1},\ldots,\mathbf{u_n}$ be unit vectors on which the vector chromatic number is attained, i.e. $$\langle \mathbf{u_i},\mathbf{u_j}\rangle= \frac{-1}{\chi_{\textrm{v}}-1}$$ whenever $i\neq j\in h$ for some $h\in H$. Let $f:V\rightarrow \R$ be an eigenfunction corresponding to the smallest eigenvalue $\lambda_1$ and let $\mathbf{v_i}:=f(i)\mathbf{u_i}$, for $i=1,\ldots,n$. Then, by Lemma \ref{lemma:Chung},
		\begin{align*}
		\lambda_n&\ge \frac{\sum_{i,j=1}^n\sqrt{\frac{\deg(i)}{\deg(j)}}L_{i,j}\langle \mathbf{v_i},\mathbf{v_j}\rangle}{\sum_{i=1}^n \|\mathbf{v_i}\|^2}
		\\&= \frac{\sum_{i\ne j}\sqrt{\frac{\deg(i)}{\deg(j)}}L_{i,j}f(i)f(j)\langle \mathbf{u_i},\mathbf{u_j}\rangle}{\sum_i f(i)^2} +1
		\\&=\frac{-1}{\chi_{\textrm{v}}-1}\frac{\sum_{i\ne j}f(i)f(j)\sqrt{\frac{\deg(i)}{\deg(j)}}L_{i,j}}{\sum_{i\in V}f(i)^2}+1
		\\&= \frac{-1}{\chi_{\textrm{v}}-1}\frac{\sum_{i, j=1}^nf(i)f(j)\sqrt{\frac{\deg(i)}{\deg(j)}}L_{i,j}}{\sum_{i\in V}f(i)^2}+ \frac{ \chi_{\textrm{v}}}{\chi_{\textrm{v}}-1}
		\\&= \frac{-1}{\chi_{\textrm{v}}-1}\lambda_1+ \frac{ \chi_{\textrm{v}}}{\chi_{\textrm{v}}-1},
		\end{align*}
		which implies that $\chi_{\textrm{v}}\ge\frac{\lambda_n-\lambda_1}{\lambda_n-1}$. A similar argument gives $\chi_{\textrm{v}}\ge\frac{\lambda_n-\lambda_1}{1-\lambda_1}$.
	\end{proof}

	The following result is a direct consequence of Theorem \ref{thm:sandwich} and Theorem \ref{thm:vector-chromatic}.
	\begin{corollary}
		\label{coro:vector-chromatic}For any oriented hypergraph $\Gamma$,
		\begin{equation*}
		\chi(\Gamma)\geq \frac{\lambda_n-\lambda_1}{\min\{\lambda_n-1,1-\lambda_1\}}.
		\end{equation*}
	\end{corollary}

	We finish this section with some examples that discuss the tightness of the bound from Theorem \ref{thm:vector-chromatic}.
	\begin{ex}
		Let $\Gamma$ be an oriented hypergraph on $n$ nodes and one single hyperedge containing all vertices. Then,  $\chi=n$, $\lambda_1=0$ and, by \cite[Theorem 1]{Sharp}, $\lambda_n=n$. Therefore, by Theorem \ref{thm:sandwich} and Theorem \ref{thm:vector-chromatic}, it holds that
		\begin{equation*}
		n=\chi(\Gamma)\geq \chi_{\textrm{v}}(\Gamma)\geq \frac{\lambda_n-\lambda_1}{\min\{\lambda_n-1,1-\lambda_1\}}=n.
		\end{equation*}Observe that for this $\Gamma$ both inequalities become equalities, providing an example of when the bound in Theorem \ref{thm:vector-chromatic} is attained.
	\end{ex}
	\begin{ex}\label{ex:5}
		Let $\Gamma=(V,H,\psi_\Gamma)$ be such that (Figure \ref{fig:5}):
		\begin{itemize}
			\item $V=\{1,2,3\}$;
			\item $H=\{h_1,h_2,h_3\}$;
			\item $h_1$ has $1$ as input and $2$ as output;
			\item $h_2$ has $1$ as input and $3$ as output;
			\item $h_3$ has $1$, $2$ and $3$ as inputs.
		\end{itemize}Then, $\chi_{\textrm{v}}(\Gamma)=3$,  $\lambda_3=3/2$ and $\lambda_1=1/2$. Hence,
		\begin{equation*}
		\chi_{\textrm{v}}(\Gamma)=3>\frac{\lambda_3-\lambda_1}{\min\{\lambda_3-1,1-\lambda_1\}}=2.
		\end{equation*}
		Therefore, the inequality in Theorem \ref{thm:vector-chromatic} cannot be improved by using equality.
		
		\begin{figure}[h]
			\centering
			\includegraphics[width=6cm]{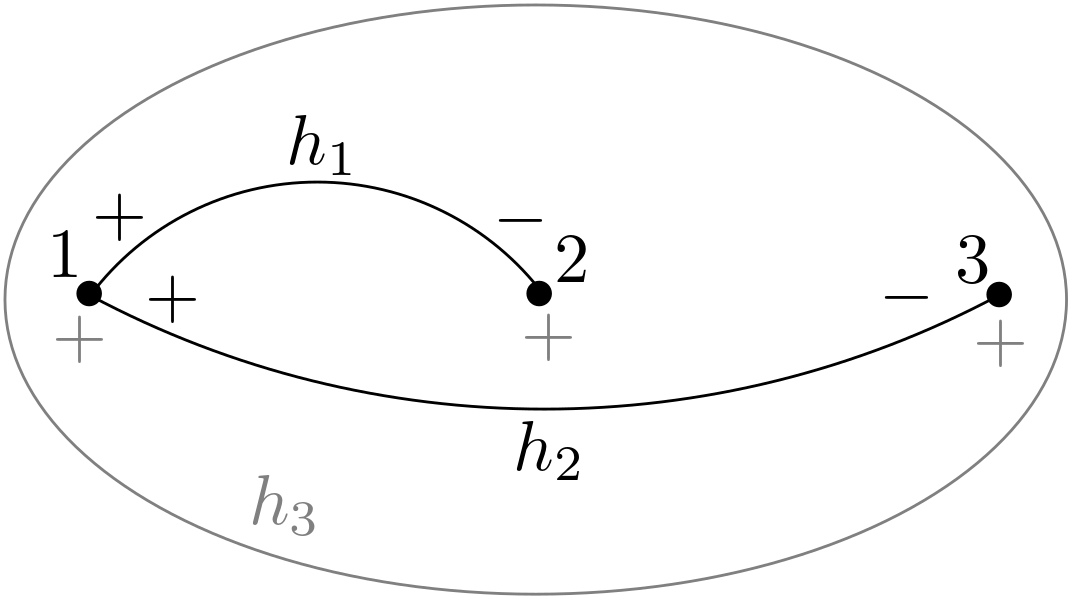}
			\caption{The oriented hypergraph from Example \ref{ex:5}.}
			\label{fig:5}
		\end{figure}
	\end{ex}

	\section{Spectral partition numbers}\label{section:spectralpart} 
	
	Finally, in this section we introduce the spectral partition numbers of an oriented hypergraph and show that they strictly relate to the coloring number discussed in Section \ref{section:vector}.
	
	Given $\lambda\geq 0$, we define the \emph{spectral partition number} of an oriented hypergraph $N_{\geq}(\lambda)$ (respectively $N_{\leq}(\lambda)$) as the smallest $k$ such that there exists a $k$--partition $V=V_1\sqcup\ldots\sqcup V_k$ of the vertex set with
	\begin{equation*}
	\lambda_1(\Gamma|_{V_l})\geq \lambda\text{ (respectively }\lambda_{\max}(\Gamma|_{V_l})\leq \lambda)\quad \text{for all }l=1,\ldots,k.
	\end{equation*}Note that $N_{\ge}(\lambda)$ is well defined for $\lambda\le1$ and it is nondecreasing in $\lambda$. Conversely, $N_{\le}(\lambda)$ is well defined for $\lambda\ge1$ and it is nonincreasing in $\lambda$. Thus, both $N_{\ge}(\lambda)$ and  $N_{\le}(\lambda)$ reach their maxima at $\lambda=1$.

	The following result shows that $N_{\geq}(1)=N_{\leq}(1)$ and, in addition, it relates the spectral partition numbers with the coloring number.
	\begin{proposition}\label{prop:chiN}
		For any oriented hypergraph, $N_{\geq}(1)=N_{\leq}(1)\leq \chi$. For graphs, $N_{\geq}(1)=N_{\leq}(1)= \chi$.
	\end{proposition}
	\begin{proof}We first prove the inequality for the general case of oriented hypergraphs. As shown in \cite{MulasZhang}, the  sum of the eigenvalues of a hypergraph equals the number of vertices. Therefore, given a $k$--partition $V=\sqcup_l V_l$,
		\begin{align*}
		\lambda_1(\Gamma|_{V_l})\geq 1\quad \text{for all } l \,&\Longleftrightarrow \lambda_1(\Gamma|_{V_l})=\ldots=\lambda_{\max}(\Gamma|_{V_l})=1\quad \text{for all } l\\
		&\Longleftrightarrow\lambda_{\max}(\Gamma|_{V_l})\leq 1\quad \text{for all } l\\
		&\Longleftrightarrow L(\Gamma|_{V_l})=\id \quad \text{for all } l\\
		&\Longleftrightarrow A(\Gamma|_{V_l})=\mathbf{0},
		\end{align*}
		where $\mathbf{0}$ is the all--zero matrix of corresponding size. Since a vertex partition given by coloring classes always satisfies $A(\Gamma|_{V_l})=\mathbf{0}$, this shows that $N_{\geq}(1)=N_{\leq}(1)\leq \chi$.
		
		For the case of graphs, we known that $A(\Gamma|_{V_l})=\mathbf{0}$ if and only if $\Gamma|_{V_l}$ has no edges, therefore in this case the partition $V=\sqcup_l V_l$ gives a partition into coloring classes, implying that $N_{\geq}(1)=N_{\leq}(1)= \chi$.
	\end{proof}
	
	The following example shows that the equality $N_{\geq}(1)=N_{\leq}(1)= \chi$ does not hold for general oriented hypergraphs.

	\begin{ex}\label{ex:3}
		Let $\Gamma$ be an oriented hypergraph with vertices $\{1, 2, 3\}$ and  hyperedges $\{h_1, h_2\}$ such that $h_1$ has $1$ as input and $2$ as output, while $h_2$ has $1$ and $2$ as inputs and $3$ as output (Figure \ref{fig:3}). Then,  
		\begin{itemize}
			\item $\lambda_1(\Gamma|_{\{1,2\}})=\lambda_2(\Gamma|_{\{1,2\}})=\lambda_1(\Gamma|_{\{1\}})=\lambda_1(\Gamma|_{\{2\}})=\lambda_1(\Gamma|_{\{3\}})=1$;
			\item $\lambda_1(\Gamma|_{\{1,3\}})=\lambda_1(\Gamma|_{\{2,3\}})=1-\frac{1}{\sqrt{2}}$;
			\item $\lambda_2(\Gamma|_{\{1,3\}})=\lambda_2(\Gamma|_{\{2,3\}})=1+\frac{1}{\sqrt{2}}$;
			\item  $\lambda_1(\Gamma)=0$, $\lambda_2(\Gamma)=1$; $\lambda_3(\Gamma)=2$
		\end{itemize}
		
		and it holds that $\chi(\Gamma)=3>2=N_{\le }(1)=N_{\ge}(1)$. 
		
		\begin{figure}[h]
			\centering
			\includegraphics[width=7cm]{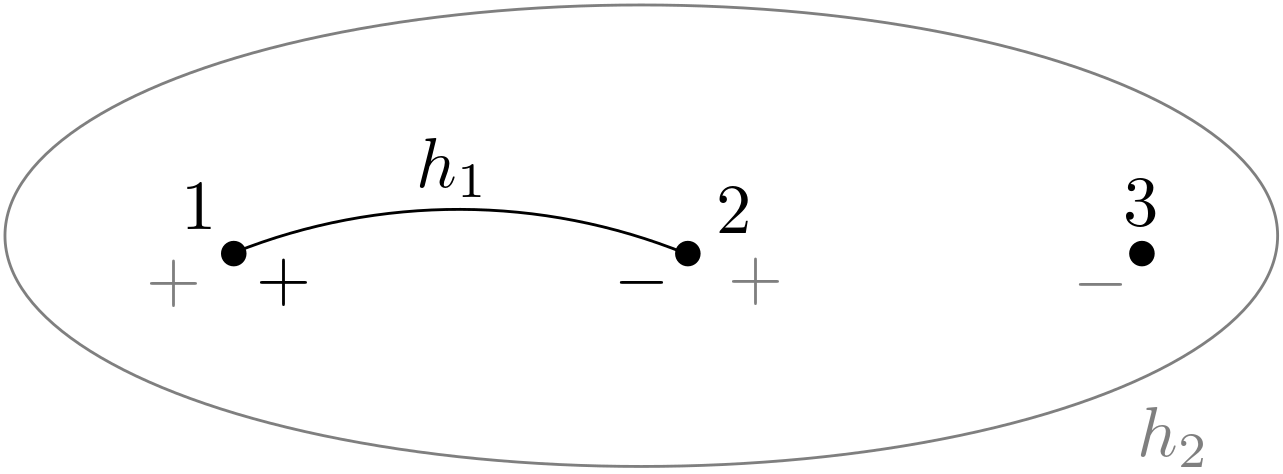}
			\caption{The oriented hypergraph from Example \ref{ex:3}.}
			\label{fig:3}
	\end{figure}\end{ex}

	Our next result lower bounds  both spectral partition numbers in terms of the smallest and largest eigenvalues of the normalized Laplacian matrix of an oriented hypergraph $\Gamma$.
	
	\begin{theorem}\label{thm:last} Let $\Gamma$ be an oriented hypergraph with eigenvalues of the normalized Laplacian matrix $\lambda_1\leq\ldots\leq \lambda_n$. Then, for any $\lambda\geq 0$,
		\begin{equation*}
		N_{\leq}(\lambda)\ge\frac{\lambda_n-\lambda_1}{ \lambda-\lambda_1}\quad\text{and}\quad N_{\geq}(\lambda)\ge\frac{\lambda_n-\lambda_1}{ \lambda_n-\lambda}.
		\end{equation*}
	\end{theorem}
	
	\begin{proof} We focus on showing the first inequality, since the second one follows by an analogous argument. Let $k:=N_{\leq}(\lambda)$ and let $V=V_1\sqcup\ldots\sqcup V_k$ a $k$--partition such that $\lambda_1(V_l)\geq\lambda$ for all $l=1,\ldots,k$. We use the notation $i\sim j$ (respectively $i\nsim j$) provided $i,j\in V$ belong to the same $V_m$ (respectively $i,j\in V$ belong to different sets of the $k$--partition).
		
		Let $\mathbf{u_1},\ldots, \mathbf{u_k}$ be the vertices of a $(k-1)$--dimensional regular simplex centered at $\mathbf{0}$ in $\mathbb{R}^n$. Then, $\langle \mathbf{u_l}, \mathbf{u_p}\rangle=-1/(k-1)$ whenever $p\ne m$. Let $f$ be an eigenfunction for $\lambda_n$ and, given $i\in V$, let $\mathbf{v_i}:=f(i)\cdot \mathbf{u_m}$ provided $i\in V_m$. Then,
		
		\begin{align*}
		\lambda_1&\le \frac{\sum_{i,j\in V}f(i)f(j)\sqrt{\frac{\deg(i)}{\deg(j)}}L_{ij}\langle \mathbf{u_i},\mathbf{u_j}\rangle}{\sum_{i\in V}f(i)^2\|\mathbf{u_i}\|^2}
		\\&=-\frac{1}{k-1} \frac{\sum_{i\nsim j}f_if_j\sqrt{\frac{\deg(i)}{\deg(j)}}L_{ij}}{\sum_{i\in V}f_i^2}+\frac{\sum_{i\sim j}f_if_j\sqrt{\frac{\deg(i)}{\deg(j)}}L_{ij}}{\sum_{i\in V}f_i^2}
		\\&=-\frac{1}{k-1}\frac{\sum_{i,j}f_if_j\sqrt{\frac{\deg(i)}{\deg(j)}}L_{ij}}{\sum_{i\in V}f_i^2}+\frac{k}{k-1}\frac{\sum_{i\sim j}f_if_j\sqrt{\frac{\deg(i)}{\deg(j)}}L_{ij}}{\sum_{i\in V}f_i^2}
		\\&\le -\frac{1}{k-1}\lambda_n + \frac{k}{k-1}\lambda
		\end{align*}which implies that $k\geq (\lambda_n-\lambda_1)/(\lambda-\lambda_1).$
		
	\end{proof}
	
	\begin{ex}
		Consider the same oriented hypergraph $\Gamma$ as in Example \ref{ex:3}. Then, 
		\begin{equation*}
		\frac{\lambda_n-\lambda_1}{\min\{\lambda_n-1,1-\lambda_1\}}=2=N_{\le }(1)=N_{\ge}(1).
		\end{equation*}
		Therefore, $\Gamma$ provides an example for which the bounds from Theorem \ref{thm:last} are sharp.
	\end{ex}

	\section*{Acknowledgments}
	The authors are grateful to the anonymous referee for the comments and suggestions that have greatly improved the first version of this paper. The research of A. Abiad is partially supported by the FWO grant 1285921N.


\begin{thebibliography}{99}
		
		\bibitem{AndreottiMulas} E. Andreotti and R. Mulas, Spectra of Signless Normalized Laplace Operators for Hypergraphs, \emph{arXiv preprint}, arXiv:2005.14484v1.
		
		\bibitem{BDOV14} C. Bachoc, E. DeCorte, F.M. de Oliveira Filho and F. Vallentin, Spectral bounds for the independence ratio and the chromatic number of an operator, \emph{Israel J. Math.} 202 (2014), 227--254.
		
		\bibitem{hyp-ind1} A.E. Balobanov and D.A. Shabanov, On the Number of Independent Sets in Simple Hypergraphs, \emph{Math. Notes} 103 (2018), 33--41.
		
		\bibitem{hyp-ind2}J. Balogh, B. Bollobas and B. Narayanan, Counting independent sets in regular hypergraphs, \emph{Journal of Combinatorial Theory, Series A} 180 (2021), 105405.
		
		
		\bibitem{bdz2005} M. Beis, W. Duckworth and M. Zito, Large k-independent sets of regular graphs, \emph{Electron. Notes Discrete Math.} 19 (2005), 321--327.
		
		\bibitem{BH2012} A.E. Brouwer and W.H. Haemers, \emph{Spectra of Graphs}, Springer, New York (2012).
		
		\bibitem{CLRRW2018}G. Chen, V. Liu, E. Robinson, L.J. Rusnak, K. Wang, A characterization of oriented hypergraphic Laplacian and adjacency matrix coefficients,  \emph{Linear Algebra Appl.} 556 (2018), 323--341.
		
		\bibitem{CRRY2015}V. Chen, A. Rao, L.J. Rusnak and A. Yang, A characterization of oriented hypergraphic balance via signed weak walks, \emph{Linear Algebra Appl.} 485 (2015), 442--453.
		
		\bibitem{hyp-ind3}E. Cohen, W. Perkins, M. Sarantis and P. Tetali, On the Number of Independent Sets in Uniform, Regular, Linear Hypergraphs, \emph{arXiv preprint}, arXiv:2001.00653.
		
		\bibitem{hyp-ind4}J. Cooper and D. Mubayi, Sparse hypergraphs with low independence number, \emph{Combinatorica} 37 (2017), 31--40.
		
		\bibitem{C71}D. M. Cvetkovi\'c, Graphs and their spectra, {\em Publ. Elektrotehn. Fak. Ser. Mut. Fiz.} 354-356 (1971), 1--50.
		
		\bibitem{DR2019} L. Duttweiler and N. Reff, Spectra of cycle and path families of oriented hypergraphs, \emph{Linear Algebra Appl.} 578 (2019), 251--271.
		
		\bibitem{Erdos}P. Erdős and A. Hajnal, On chromatic number of graphs and set-systems, \emph{Acta Math. Acad. Sci. Hungar.} 17 (1966), 61--99.
		
		\bibitem{hyp-ind5}J. Fox and X. He, Independent sets in hypergraphs with a forbidden link, \emph{Proceedings of the London Mathematical Society} (2021), https://doi.org/10.1112/plms.12400.
		
		\bibitem{GR2001} C. Godsil and G. Royle, \emph{Algebraic Graph Theory}, Springer-Verlag, New York (2001).
		
		\bibitem{GRR2019} W. Grilliette, J. Reynes and L.J. Rusnak, Incidence Hypergraphs: Injectivity, Uniformity, and Matrix-tree Theorems, \emph{arXiv preprint}, arXiv:1910.02305 (2019).
		
		\bibitem{GR2020} W. Grilliette and L.J. Rusnak, Incidence Hypergraphs: Box Products \& the Laplacian, \emph{arXiv preprint}, arXiv:2007.01842 (2019).
		
		
		\bibitem{matrixanalysis}R.A. Horn and C.R. Johnson, Matrix Analysis, \emph{Cambridge University Press} (2013), second edition.
		
		\bibitem{H70}A.J. Hoffman, 
		On eigenvalues and colorings of graphs, Graph Theory and Its Applications, Academic Press, New York (1970), 79--91.
		
		
		\bibitem{H2020} H. Huang, Induced subgraphs of hypercubes and a proof of the Sensitivity Conjecture, \emph{Ann. of Math.} 190 (2019), 949--955.
		
		
		\bibitem{Hypergraphs}  J. Jost and R. Mulas, Hypergraph Laplace operators for chemical reaction networks, \emph{Adv. Math.} 351 (2019), 870--896.
		
		\bibitem{pLaplacians1} J. Jost, R. Mulas and D. Zhang, $p$-Laplace Operators for Chemical Hypergraphs, \emph{Vietnam Journal of Mathematics}, To appear (2022).
		
		
		\bibitem{KMS98} D. Karger, R. Motwani and M. Sudan, Approximate graph Sharp by semidefinite programming, \emph{J. ACM} 45 (1998), 246--265.
		
		\bibitem{KR2019} O. Kitouni and N. Reff, Lower bounds for the Laplacian spectral radius of an oriented hypergraph, \emph{Australas. J. Combin.} 74(3) (2019), 408--422.
		
		
		\bibitem{Knuth}D.E. Knuth, The sandwich theorem, \emph{Electron. J. Combin.} 1:A1 (1994).
		
		\bibitem{hyp-ind6}A. Kostochka, D. Mubayi and J. Verstraëte, On independent sets in hypergraphs, \emph{Random Structures Algorithms} 44(2) (2014), 224--239.
		
		\bibitem{hyp-ind7}S.J. Lee and H. Lefmann, The independence number of non-uniform uncrowded hypergraphs and an anti-Ramsey type result, \emph{Discrete Math.} 343(9) (2020), 111964.
		
		\bibitem{MulasZhang} 
		R. Mulas and D. Zhang, Spectral theory of Laplace operators on oriented hypergraphs, \emph{Discrete Math.} 344(6) (2021), 112372, DOI: 10.1016/j.disc.2021.112372.
		
		\bibitem{Sharp} R. Mulas. Sharp bounds for the largest eigenvalue, \emph{Math. Notes}, 109(1) (2021), 102–109, DOI: 10.1134/S0001434621010120.
		
		\bibitem{Classes} R. Mulas, Spectral classes of hypergraphs, \emph{Australas. J. Combin.}, 79(3) (2021), 495–514.
		
		\bibitem{MKJ} R. Mulas, C. Kuehn and J. Jost,  Coupled Dynamics on Hypergraphs: Master Stability of Steady States and Synchronization, \emph{Phys. Rev. E} 101(6) (2020), 062313.
		
		
		\bibitem{hyp-ind8}S. Pirzada, C. Tariq, Z. Guofei and I. Antal, On independence numbers of regular hypergraphs, \emph{Acta Univ. Sap. Informatica} 6 (2014), 132--158.
		
		\bibitem{Reff2014}N. Reff, Spectral properties of oriented hypergraphs, \emph{Electron. J. Linear Algebra} 27 (2014).
		
		\bibitem{Reff2016} N. Reff, Intersection graphs of oriented hypergraphs and their matrices, \emph{Australas. J. Combin.} 65(1) (2016), 108--123.
		
		\bibitem{ReffRusnak} N. Reff and L.J. Rusnak,  An oriented hypergraphic approach to algebraic graph theory, \emph{Linear Algebra Appl.} 437 (2012), 2262--2270.
		
		\bibitem{RRSS2017} E. Robinson, L.J. Rusnak, M. Schmidt and P. Shroff, Oriented hypergraphic matrix-tree type theorems and bidirected minors via Boolean order ideals, \emph{J. Algebraic Combin.} (2017).
		
		\bibitem{Rusnak2013} L.J. Rusnak, Oriented Hypergraphs: Introduction and Balance, \emph{Electron. J. Combin.} 20(3) (2013).
		
		\bibitem{Shi92} C.-J. Shi, A signed hypergraph model of the constrained via minimization problem, \emph{Microelectron. J.} 23(7) (1992), 533--542.
		
		\bibitem{SunDas2020}S. Sun and K.-C. Das, Normalized Laplacian eigenvalues with chromatic number and independence number of graphs, \emph{Linear Multilinear Algebra} 68(1) (2020), 63--80.
		
		\bibitem{hyp-ind9}T. Thiele, A lower bound on the independence number of arbitrary hypergraphs, \emph{J. Graph Theory} 30(3) (1999), 213--221.
		
		\bibitem{hyp-ind10}G. Zhou and Y. Li, Independence numbers of hypergraphs with sparse neighborhoods, \emph{European J. Combin.} 25(3) (2004), 355--362.
		
	\end{thebibliography}
\end{document}